\documentclass{article}

\usepackage{multicol}
\usepackage{amssymb,latexsym,amsmath,amsthm,amsfonts} 
\usepackage{graphicx}
\usepackage{tikz}
\usepackage{enumitem}
\usepackage{verbatim}
\usepackage{color}
\usepackage{xcolor}
\usepackage{mathrsfs}
\usepackage[normalem]{ulem}

\newtheorem{thm}{Theorem}

\newtheorem{cnj}[thm]{Conjecture}

\newtheorem{lem}[thm]{Lemma}

\newtheorem{prob}{Problem}

\usepackage{subfiles}
\usepackage{float}
\usepackage{mymacros}

\newcommand\omicron{o}
\pgfdeclarelayer{background}
\pgfdeclarelayer{foreground}
\pgfsetlayers{background,main,foreground}

\newcommand{\C}{\mathcal{C}}

\newcommand{\B}{\mathcal{B}}
\newcommand{\R}{\mathcal{R}}

\begin{document}

\title{On intersecting families of subgraphs of perfect matchings}

\author{
Melissa Fuentes \footnotemark[1]\ 
\and
Vikram Kamat \footnotemark[2]
}
\date{}
\maketitle

\footnotetext[1]{
Department of Mathematics \& Statistics, 
Villanova University, 
Villanova, PA, USA, 
\texttt{melissa.fuentes@villanova.edu}
}

\footnotetext[2]{
Department of Mathematics \& Statistics, 
Villanova University, 
Villanova, PA, USA, 
\texttt{vikram.kamat@villanova.edu}
}


\begin{abstract}
\noindent
   The seminal Erd\H{o}s--Ko--Rado (EKR) theorem states that if $\mathcal{F}$ is a family of $k$-subsets of an $n$-element set $X$ for $k\leq n/2$ such that every pair of subsets in $\mathcal{F}$ has a nonempty intersection, then $\mathcal{F}$ can be no bigger than the trivially intersecting family obtained by including all $k$-subsets of $X$ that contain a fixed element $x\in X$. This family is called the \textit{star} centered at $x$. In this paper, we formulate and prove an EKR theorem for intersecting families of subgraphs of the perfect matching graph, the graph consisting of $n$ disjoint edges. This can be considered a generalization not only of the aforementioned EKR theorem but also of a \textit{signed} variant of it, first stated by Meyer \cite{Meyer}, and proved separately by Deza--Frankl \cite{DezFra} and Bollob\'as--Leader \cite{BollLead}. The proof of our main theorem relies on a novel extension of Katona's beautiful \textit{cycle method}.
\end{abstract}

\section{Introduction}
For a finite set $X$ containing $n$ elements, where $n$ is a positive integer, let $2^{X}$ and $\binom{X}{r}$ denote the family of all subsets and $r$-subsets of $X$, respectively. For any $\cG\subseteq 2^{X}$ and $x\in X$, let $\cG_x$ be all sets in $\cG$ that contain $x$. We call $\cG_x$ the \textit{star in $\cG$ centered at $x$}. 
\\
\\
A family of subsets $\cF$ is \textit{intersecting} if $F\cap G\neq \emptyset$ for $F, G\in \cF$.  A classical result of Erd\H{o}s, Ko and Rado \cite{EKR} states that if $\cF\subseteq \binom{X}{r}$ is intersecting for $r\leq n/2$, then $|\cF|\leq \binom{n-1}{r-1}$. Moreover, if $r<n/2$, equality holds if and only if $\cF=\binom{X}{r}_x$ for some $x\in X$. 
\\
\\
The Erd\H{o}s--Ko--Rado theorem is one of the fundamental theorems in extremal combinatorics, and has been generalized in many directions. For instance, an ``EKR-type'' problem can be defined on a class of mathematical objects with some natural notion of pairwise intersection on objects in this class; a typical result along these lines involves finding a best possible upper bound on the size of the largest intersecting subfamily within this class. Furthermore, it can often be shown that the \textit{extremal} structures -- typically analogous to the star structure defined above  -- are unique. Indeed, for $\mathcal{G}\subseteq 2^X$, we say that $\cG$ is \textit{EKR} if there exists $x\in X$ such that for any intersecting subfamily $\cF\subseteq \cG$, $|\cF|\leq |\cG_x|$. Furthermore, we say that $\cG$ is \textit{strongly EKR} if \emph{every} maximum intersecting subfamily of $\cG$ is a star in $\mathcal{G}$. For example, it is an easy exercise to verify that if $\cG=2^X$, then $\cG$ is EKR but not strongly EKR. 
\\
\\
EKR results along these lines have been proved for, among other objects, permutations, vector spaces, set partitions and families of independent sets for certain classes of graphs. We refer the reader to \cite{GodMea}, and the references contained within, for more details on these and other generalizations inspired by the theorem. 
\\
\\
In this paper, we consider an EKR-type problem for families of induced subgraphs in the \textit{perfect matching graph}, which we define in the next section.

\subsection{Induced subgraphs of perfect matchings}
Let $G=(V,E)$ be a graph with vertex set $V=V(G)$ and edge set $E=E(G)$ containing (undirected) edges between pairs of vertices. An \textit{induced} subgraph $H=(V',E')$ of $G$ is a subgraph with $V'\subseteq V$ and $E'\subseteq E$ such that for any $u,v\in V'$, $\{u,v\}\in E'$ if and only if $\{u,v\}\in E$. Also, for positive integers $i,j, n$ and $1\leq i\leq j\leq n$ let $[i,j]=\{i,i+1,\ldots,j\}$. Let $[n]=[1,n]$.
\\
\\
For $n\geq 1$, we define the \textit{perfect matching graph}, denoted by $M_n$, as the graph that consists of $n$ pairwise disjoint copies of the complete graph $K_2$. Let $E(M_n)=\{ e_1, e_2, \ldots, e_n\}$. For each $i \in [n]$, let $e_i=\{l_i,r_i\}$, where $l_i$ and $r_i$ are the vertex endpoints of the edge $e_i$. Then $V(M_n)=L\cup R$, where $L=\{l_1,l_2,\ldots,l_n\}$, and $R=\{r_1,r_2,\ldots,r_n\}$.
\\
\\
For $r\geq 1$, denote the family of all induced subgraphs of $M_n$ containing $r$ vertices by $\mathcal{H}^{(r)}(n)$. For $s,p\geq 0$ and $2p+s\geq 1$, let $\mathcal{H}^{(p,s)}(n)=\{V(H):H\in \mathcal{H}^{(2p+s)}(n), |E(H)|=p\}$; that is, $\mathcal{H}^{(p,s)}(n)$ is the family of vertex subsets of all induced subgraphs of $M_n$ that consist of $p$ disjoint edges and $s$ isolated vertices. Though members of this family are vertex subsets of $M_n$, we will refer to them as subgraphs; additionally, for brevity, a member $V(H)\in \mathcal{H}^{(p,s)}(n)$ will be denoted by $H$. Note that 

\begin{equation}\label{eq1}
|\mathcal{H}^{(p,s)}(n)|=\binom{n}{p}\binom{n-p}{s}2^s.
\end{equation}
Note also that for any $x\in V(M_n)$, the cardinality of the star in $\mathcal{H}^{(p,s)}(n)$ centered at $x$ is given as follows:
\begin{equation} \label{eq2}
|\mathcal{H}^{(p,s)}_x(n)|=\binom{n-1}{p-1}\binom{n-p}{s}2^s+\binom{n-1}{p}\binom{n-p-1}{s-1}2^{s-1}=\displaystyle (2p+s)\frac{(n-1)!}{p!s!(n-p-s)!}2^{s-1}.
\end{equation}
Finally, using Equations \ref{eq1} and \ref{eq2}, we observe that
\begin{equation}\label{eq3}
(2n)|\mathcal{H}^{(p,s)}_x(n)|=(2p+s)|\mathcal{H}^{(p,s)}(n)|.
\end{equation}
\\
\\
We are now ready to state our main results.

\section{Main results \& a conjecture}
Our main results are motivated by the following conjecture that we propose as an extension of two distinct EKR theorems.
\begin{cnj}\label{cj}
     For non-negative integers $s,p$ and $1\leq 2p+s\leq n$, $\mathcal{H}^{(p,s)}(n)$ is EKR.
\end{cnj}
\noindent
Indeed, it is easy to observe that the case $s=0$ is the Erd\H{o}s--Ko--Rado theorem itself. Furthermore, the case $p=0$ is an EKR result for intersecting families of \textit{independent} sets in $M_n$, first stated in Meyer \cite{Meyer}, and proved separately by Deza--Frankl \cite{DezFra} and Bollob\'as--Leader \cite{BollLead}.
\\
\\
Our main result in this paper is the following theorem which assumes a stronger condition on $n$ (for any fixed $p$ and $s$) as compared to Conjecture \ref{cj}.
\begin{thm}\label{th1}
    For $s, p\geq 1$ and  $n\geq 2(p+s)$, $\mathcal{H}^{(p,s)}(n)$ is EKR; it is \textit{strongly} EKR when $n>2(p+s)$.
\end{thm}
\noindent
We also prove the following result, which proves Conjecture \ref{cj} for $n=2p+s$ and presents a strengthening of Theorem \ref{th1} for the special case $s=1$. 
\begin{thm}\label{th2}
Let $s,p\geq 1$.
\begin{enumerate}
  \item \label{two1} $\mathcal{H}^{(p,s)}(2p+s)$ is EKR but not strongly EKR. 
  \item \label{two2} If $s=1$ and $n>2p+1$, then $\mathcal{H}^{(p,s)}(n)$ is strongly EKR.
\end{enumerate}
\end{thm}
\noindent
The proof of Theorem \ref{th1} employs a novel extension of Katona's famous cycle method. In contrast to the cyclic orderings employed in \cite{BollLead} and \cite{FHK} to prove their respective EKR theorems, it is more natural here to consider cyclic permutations of the edges in $M_n$. However, to account for the fact that subgraphs contain edges -- i.e. both vertices -- or just one vertex from an edge, we have to expand the definition of an \textit{interval} in these cyclic orders. We do so by considering two different types of intervals within a given cyclic order. 
\\
\\
We also mention here that the variant of the cycle method formulated in \cite{FHK} could be used to prove a weaker version of Theorem \ref{th1}; more precisely, it would require a stronger condition on $n$, namely $n\geq 2(2p+s)$. Our result, on the other hand, also handles the cases when $2p+2s\leq n<4p+2s$.
\\
\\
The rest of the paper is organized as follows. In Section \ref{cycle}, we define all notation necessary for our generalization of Katona's method and lemma, and use it to prove the tight upper bound in Theorem \ref{th1}. In Section \ref{characterization}, we prove the uniqueness of the extremal structures for Theorem \ref{th1}. In Section \ref{SOne}, we prove Theorem \ref{th2}. In Section \ref{future}, we outline some directions for further research.

\section{Katona's cycle and the upper bound in Theorem \ref{th1}}\label{cycle}
For the purposes of this proof, for any edge $e_i\in E(M_n)$, we denote $e^0_i=l_i$ and $e^1_i=r_i$. Let $S_n$ denote the set of all permutations $\sigma$ of $[n]$, and let $\{0,1\}^n$ denote the set of all sequences $\tau=(\tau_1, \ldots, \tau_n)$ where $\tau_i\in \{0,1\}$ for each $1\leq i \leq n$. For each choice of $\sigma\in S_n$ and $\tau\in \{0,1\}^n$, we can define a cyclic order $C_\sigma^\tau$ of $E(M_n)$ as follows: $C_\sigma^\tau=\left(e^{(\tau_1)}_{\sigma(1)}, e^{(\tau_2)}_{\sigma(2)}, \ldots, e^{(\tau_n)}_{\sigma(n)}\right)$, where for each $i \in[n]$, $e_{\sigma(i)}^{(\tau_i)}=\left(e_{\sigma(i)}^{\tau_i},e_{\sigma(i)}^{\tau_i+1}\right)$, where $\tau_i+1$ is computed modulo $2$. Rotating a given cyclic order along the $n$ positions corresponding to the $n$ edges gives $n-1$ other cyclic orders that we say are \textit{equivalent} to the original cyclic order. To account for this notion of equivalence, we only consider cyclic orders $C_\sigma^\tau$ where $\sigma\in S_n$ is a permutation with $\sigma(n)=n$. Indeed, in the proof of the characterization of the extremal structures in Section \ref{characterization}, we implicitly identify $\sigma$ as a permutation in $S_{n-1}$.  Let $\mathscr{C}_n$ denote the set of these orders. It is clear to see that $|\mathscr{C}_n|=(n-1)!2^n$.
\\
\\
For a cyclic order $C_\sigma^\tau$, positive integers $p$, $s$ (with $p+s\leq n$), and $i\in [n]$, we define a $B$-interval beginning at the edge in position $i$ as follows: $B_i=B_i(\sigma,\tau,p,s)=\left(e^{(\tau_i)}_{\sigma(i)},e^{(\tau_{i+1})}_{\sigma(i+1)},\ldots,e^{(\tau_{i+p-1})}_{\sigma(i+p-1)},e^{\tau_{i+p}}_{\sigma(i+p)},\ldots,e^{\tau_{i+p+s-1}}_{\sigma(i+p+s-1)}\right)$. 
\\
\noindent
Similarly, we define an $R$-interval beginning at the edge in position $i$ as follows: $R_i=R_i(\sigma,\tau,p,s)=\left(e^{\tau_i+1}_{\sigma(i)},e^{\tau_{i+1}+1}_{\sigma(i+1)},\ldots,e^{\tau_{i+s-1}+1}_{\sigma(i+s-1)},e^{(\tau_{i+s})}_{\sigma(i+s)},\ldots,e^{(\tau_{i+p+s-1})}_{\sigma(i+p+s-1)}\right)$. Note that addition here is carried out modulo $n$, so for any $j$, $i+j=i+j-n$ if $i+j>n$. For fixed $p$ and $s$, each cyclic order thus has $n$ $B$-intervals and $n$ $R$-intervals. 
\\
\\
More informally, both $B$-intervals and $R$-intervals cover $p+s$ consecutive edges in a given cyclic order. However, reading clockwise, a $B$-interval contains both vertices from its first $p$ edges and only the vertex in the first position (in the cyclic order) from each of the last $s$ edges. Similarly, a $R$-interval contains only the vertex in the second position from each of its first $s$ edges, and both vertices from its last $p$ edges. When the context is clear (with regard to $\sigma$, $\tau$, $p$ and $s$), we will refer to $B$-intervals and $R$-intervals by their starting positions alone.  
\\
\\
Before we proceed to a proof of the upper bound in Theorem \ref{th1}, we illustrate the above definitions with an example. Let $n=6$, $p=1$, and $s=2$. Given permutation $\sigma=(\sigma(1),\ldots,\sigma(6))=(5,3,2,1,4,6)$ and $\tau=(0,1,1,0,1,0)$, the corresponding cylic order is $C_\sigma^\tau=((l_5,r_5),(r_3,l_3),(r_2,l_2),(l_1,r_1),(r_4,l_4),(l_6,r_6))$. The $B$-intervals $B_3$ and $B_5$ in this order are $((r_2,l_2),l_1,r_4)$ and $((r_4,l_4),l_6,l_5)$ respectively. Similarly, the $R$-intervals $R_2$ and $R_6$ in this order are $(l_3,l_2,(l_1,r_1))$ and $(r_6,r_5,(r_3,l_3))$ respectively.
\\
\\ 
We also clarify a mild abuse of notation regarding intervals that we will frequently employ for convenience of exposition. First, it is clear that both $B$ and $R$ intervals can be regarded as ordered sets of $2p+s$ vertices in $V(M_n)$. Secondly, both $B$ and $R$ intervals naturally correspond to members of $\mathcal{H}^{(p,s)}(n)$. In view of this correspondence, a $B$ or $R$-interval will often be identified with and refer to the (unordered vertex subsets of the) induced subgraph in $\mathcal{H}^{(p,s)}(n)$ that contain the $p$ edges and $s$ singleton vertices in that interval. This equivalence will typically be implied in statements that involve equality between or intersection of two intervals.
\\
\\
We now prove the upper bound in Theorem \ref{th1}. For the remainder of this section, let $p,s\geq 1$, $n\geq 2(p+s)$, and suppose that $\mathcal{F}\subseteq \mathcal{H}^{(p,s)}(n)$ is intersecting. 
\\
\\
\noindent
For a given cyclic arrangement $C_\sigma^\tau$, let $\mathcal{B}_\sigma^\tau$ and $\mathcal{R}_\sigma^\tau$ be the subfamilies of all members in $\mathcal{F}$ that can be ordered as $B$-intervals and $R$-intervals in $C_\sigma^\tau$ respectively. Let $\mathcal{F}_\sigma^\tau=\mathcal{B}_\sigma^\tau \cup \mathcal{R}_\sigma^\tau$. The following lemma, key to establishing the upper bound in Theorem \ref{th1}, is an analog of Katona's lemma used in his proof of the original EKR theorem. 
\begin{lem}\label{lem1}
$|\mathcal{F}_\sigma^\tau|\leq 2p+s$.
\end{lem}
\begin{proof}
    We bound $\mathcal{B}=\mathcal{B}_\sigma^\tau$ and $\mathcal{R}=\mathcal{R}_\sigma^\tau$ separately. Note first that by using Katona's original lemma, since $n\geq 2(p+s)$, the bounds $|\mathcal{B}|\leq p+s$ and $|\mathcal{R}|\leq p+s$ are immediate. We may assume that $|\mathcal{B}|\geq 2$, otherwise we get $|\mathcal{F}_\sigma^\tau|\leq p+s+1\leq 2p+s$, as $p\geq 1$. 
    \\
    \\
    Let $1\leq k=\textrm{min }_{1\leq i, j\leq n} |B_i\cap B_j|$. Without loss of generality, suppose that $B_k, B_{p+s}\in \mathcal{B}$. Clearly $|B_k\cap B_{p+s}|=k$. Now, if $1\leq i<k$ or $2(p+s)\leq i\leq n$, then $B_i\notin \mathcal{B}$. This is because in the former case, we get $1\leq |B_i\cap B_{p+s}|<k$, which contradicts the minimality of $k$ while in the latter case, we get $B_i\cap B_{p+s}=\emptyset$. Similarly, if $p+s<i\leq p+s+k-1$, then $1\leq |B_i\cap B_k|<k$, again contradicting the minimality of $k$ and thus implying that $B_i\notin \mathcal{B}$. Finally, for $k<i<p+s$, at most one out of the disjoint pair $(B_i,B_{p+s+i})$ can be in $\mathcal{B}$. This gives us $|\mathcal{B}|\leq p+s-(k-1)$. Indeed, using Katona's bound, we can assume that $k\leq s+1$.
    \\
    \\
    To bound $\mathcal{R}$, we begin by noting that for any $2(p+s)\leq i\leq n$, $R_i\cap B_{p+s}=\emptyset$, implying that $R_i\notin \mathcal{R}$. Additionally, if $2p+s\leq i\leq 2(p+s)-1$, then $R_i\notin \mathcal{R}$, as $R_i\cap B_{p+s}=\emptyset$. Similarly, for any $k+p\leq i\leq k+p+s-1$,  $R_i\notin \mathcal{R}$ as $R_i\cap B_k=\emptyset$. Finally, for each $1\leq i\leq p+k-1$, at most one out of the disjoint pair $(R_i,R_{p+s+i})$ can be a member of $\mathcal{R}$, which implies a bound of $|\mathcal{R}|\leq p+k-1$. Alongside the bound for $|\mathcal{B}|$, this implies the required bound for $\mathcal{F}_\sigma^\tau$.
\end{proof}
\noindent
The upper bound now follows from the following double counting argument. For each $F\in \mathcal{H}^{(p,s)}(n)$, $F$ is a $B$-interval in exactly $2^p p! s! 2^{n-p-s} (n-p-s)!=2^{n-s}p!s!(n-p-s)!$ cyclic orders, and is also an $R$-interval in the same number of cyclic orders. Using Lemma \ref{lem1} and Equation \ref{eq2}, we get:
\begin{align}
|\mathcal{F}|&\leq \dfrac{(2p+s) (n-1)! 2^n}{2^{n-s+1}p!s! (n-p-s)!} \nonumber \\ \nonumber
       &=(2p+s)\dfrac{(n-1)!}{p!s!(n-p-s)!}2^{s-1}\\ \nonumber
       &=|\mathcal{H}_x^{(p,s)}(n)|,\qedhere
\end{align} for any $x\in V(M_n)$. This proves that $\mathcal{H}^{(p,s)}(n)$ is EKR for $n\geq 2(p+s)$. We now proceed to show, in the next section, that it is \textit{strongly} EKR when $n>2(p+s)$.  

\section{Characterization of extremal structures}\label{characterization}

Suppose that $n> 2(p+s)$ and that $|\mathcal{F}|=|\mathcal{H}^{(p,s)}_x(n)|$. This immediately implies that for each cyclic order $C_\sigma^\tau\in \mathscr{C}_n$, $|\mathcal{F}_\sigma^\tau|=2p+s$. Thus, we have $p\leq |\mathcal{B}_\sigma^\tau|\leq p+s$ and $p\leq |\mathcal{R}_\sigma^\tau|\leq p+s$.
\\
\\
Before we proceed with the proof of structural uniqueness of the extremal families, we recall some notation from the proof of Lemma \ref{lem1}. Given a cyclic arrangement $\C_\sigma^\tau$, let $k=k(\sigma,\tau,p,s)$ be defined as the smallest size of the intersection of two $B$-intervals in $\mathcal{B}_\sigma^\tau$. (If $|\mathcal{B}_\sigma^\tau|=1$, then $k=p+s$.) From the proof of Lemma \ref{lem1}, we may also conclude that $1\leq k\leq s+1$.
\\
\\
Our argument will demonstrate that $\mathcal{F}=\mathcal{H}^{(p,s)}_x(n)$ for some $x\in V(M_n)$, and will progress through a series of three lemmas that we state and prove below. Our first lemma characterizes the local structure of the family $\mathcal{F}$ within each cyclic arrangement, and builds on the ideas used in the proof of Lemma \ref{lem1}.
\begin{lem}\label{lem2}
    For a given cyclic arrangement $C_\sigma^\tau$, if $|\mathcal{F}_\sigma^\tau|=2p+s$, then there exists $i\in [n]$ and $k=k(\sigma,\tau,p,s)$ such that $\mathcal{B}_\sigma^\tau=\{B_i,B_{i+1},\ldots,B_{i+p+s-k}\}$ and $\mathcal{R}_\sigma^\tau=\{R_{i-(k-1)},\ldots,R_{i+p-1}\}$.
\end{lem}
\begin{proof}
Let $\mathcal{B}=\mathcal{B}_\sigma^\tau$ and $\mathcal{R}=\mathcal{R}_\sigma^\tau$. We first tackle the simple case where $|\mathcal{B}|=1$. Let $\mathcal{B}=\{B_i\}$ for some $i\in [n]$. It is clear from the argument made in the proof of Lemma \ref{lem1} that $p=1$ and $|\mathcal{R}|=p+s=s+1$. A well-known corollary of Katona's lemma \footnote{See Lemma 10 from \cite{Keevash} for a short proof of this corollary.} tells us that $\mathcal{R}=\{R_j,\ldots, R_{j+s}\}$ for some $j\in [n]$. We claim that $j+s=i$ (addition modulo $n$); indeed, if $i\in [n]\setminus [j,j+s]$, then $B_i\cap R_j=\emptyset$ and if $i\in [j,j+s-1]$, then $B_i\cap R_{j+s}=\emptyset$. This settles the case (since $k$ here is just $|B_i|=p+s=s+1$). Thus, we may assume $|\mathcal{B}|\geq 2$. Without loss of generality, suppose that $B_k, B_{p+s}\in \mathcal{B}$ with $|B_k\cap B_{p+s}|=k$.
\\
\\
Since $|\mathcal{F}_\sigma^\tau|=2p+s$, $|\mathcal{B}|=p+s-k+1$ and $|\mathcal{R}|=p+k-1$ by Lemma \ref{lem1}. Also note that by our proof of Lemma \ref{lem1}, for each $k<i<p+s$, exactly one out of the disjoint pair $(B_i, B_{p+s+i})$ must be in $\mathcal{B}$. Suppose $j$ is the smallest index such that $k<j<p+s$ and $B_j \notin \mathcal{B}$. Then $B_{j-1}, B_{p+s+j}\in \mathcal{B}$. However, as $n>2(p+s)$, $B_{j-1} \cap B_{p+s+j} =\emptyset$, a contradiction. Thus,
\[
\B=\{ B_k, B_{k+1}, B_{k+2}, \ldots, B_{p+s} \}.
\]
\noindent
Now, by our proof of Lemma \ref{lem1}, for each $1\leq i\leq p+k-1$, exactly one out of the disjoint pair $(R_i, R_{p+s+i})$ is in $\mathcal{R}$. However, since $k \leq s+1$, we have $p+k \leq p+s+1$, implying that $B_k \cap R_{p+s+1}=\emptyset$. Thus, $R_1 \in \mathcal{R}$. Finally, for $2\leq i\leq p+k-1$, let $i$ be minimum such that $R_i\notin \mathcal{R}$. Then $R_{i-1}\in \mathcal{R}$ and $R_{p+s+i}\in \mathcal{R}$. Again, as $n>2(p+s)$, $R_{i-1}\cap R_{p+s+i}=\emptyset$, a contradiction. Thus,
\[
\R=\{ R_1, R_2, \ldots, R_{p+k-1} \}.
\]
\end{proof}
\noindent
We now introduce some important notation that is required for the next lemmas. First, let $C_\iota^\omicron$ be the \textit{canonical} cyclic order, where $\iota\in S_n$ is the identity permutation, and $\omicron$ is the all-zeroes sequence (of length $n$). For any cyclic order $C_\sigma^\tau$, a \textit{transposition} $t_{i,j}$ is an operation that exchanges the positions of $e_{\sigma(i)}$ and $e_{\sigma(j)}$ in the cyclic order, while retaining the orders of the respective vertices within each of the two edges. Note that if $C_\pi^{\phi}$ is the cyclic order obtained from $C_\sigma^\tau$ using a transposition $t_{i,j}$, then $\pi=\sigma\circ (i, j)$, i.e., the transposition acts by interchanging the elements in positions $i$ and $j$. Additionally, we have $\phi_i=\tau_j$, $\phi_j=\tau_i$, and $\phi_k=\tau_k$ for all $k\in [n]\setminus \{i,j\}$. If $j=i+1$ for $1\leq i\leq n-1$, then we call $t_i=t_{i,i+1}$ an \textit{adjacent transposition}. Let $C_{\pi_i}^\phi$ denote the cyclic order obtained from $C_\sigma^\tau$ via the transposition $t_i$, where $\pi_i=\sigma\circ (i, i+1)$.
\\
\\  
Finally, for a given cyclic order $C_\sigma^\tau$ and $1\leq i\leq n$, a \textit{swap} $s_i$ is an operation that exchanges the positions of the vertices in $e_{\sigma(i)}$. Let $C_\sigma^{\tau^i}$ denote the cyclic order obtained using a swap $s_i$, where $\tau^i_i=(\tau_i+1) (\textrm{mod } 2)$ and $\tau^i_k=\tau_k$ for $1\leq k\neq i\leq n$. Note that the permutation $\sigma$ is unaffected by the swap $s_i$.
\\
\\
Each of the following two lemmas involves the use of transpositions and/or swaps on a given cyclic order to obtain a new cyclic order; to avoid ambiguity, we adopt the notation $B_i(\sigma,\tau)=B_i(\sigma,\tau,p,s)$ and $R_i(\sigma,\tau)=R_i(\sigma,\tau,p,s)$ to respectively identify $B$-intervals and $R$-intervals in a cyclic order $C_\sigma^\tau$ that begin at the edge in position $i$ of that cyclic order. 
\\
\\
The following lemma, stated and proved below, demonstrates that for each $C_\sigma^\tau\in \mathscr{C}_n$, $\mathcal{F}_\sigma^\tau$ has a ``star structure'', i.e. there exists $x\in V(M_n)$ such that $x\in \bigcap\limits_{F\in \mathcal{F}_\sigma^\tau} F$. In this case, we say that $C_\sigma^\tau$ is \textit{centered} at $x$. 
\begin{lem}\label{lem3}
For each $C_\sigma^\tau$, $k\in \{1,s+1\}$.
\end{lem}
\begin{proof}
    Consider the canonical cyclic order $C_\iota^\omicron$, and let $k=k(\iota,\omicron,p,s)$. By way of contradiction, suppose that $2\leq k\leq s$. Without loss of generality (relabeling if necessary), suppose that $\mathcal{B}_\iota^\omicron=\{B_k,B_{k+1},\ldots,B_{p+s}\}$. By Lemma \ref{lem2}, we also know that $\mathcal{R}_\iota^\omicron=\{R_1,\ldots,R_{p+k-1}\}$. Now consider the cyclic order $C_\pi^\phi$ obtained by applying the transpositions $t_{1,p+k-1}$ and $t_{p+s+k-1,2(p+s)-1}$ to the canonical cyclic order. Note that as $p\geq 1$ and $2\leq k\leq s$, we have $1<p+k-1<p+s+k-1<2(p+s)-1$. (Note also that $\phi=\omicron$.)
    \\
    \\
    Let $k'=k'(\pi,\phi,p,s)$. We now proceed to prove that $|\mathcal{F}_\pi^\phi|<2p+s$, which will complete the proof by contradiction. We begin by showing that for each $i\in [n]\setminus [k+1,p+s]$, $B_i(\pi,\phi)\notin \mathcal{B}_\pi^\phi$.
    \\
    \\
    First, note that for each $2(p+s)\leq i\leq n$, $B_i(\pi,\phi)\notin \mathcal{B}_\pi^\phi$, since $B_i(\pi,\phi)\cap B_{p+s}(\iota,\omicron)=\emptyset$. Analogously, if $2\leq i\leq k-1$, then $B_i (\pi,\phi)\cap R_{p+k-1}(\iota,\omicron)=\emptyset$, hence $B_i(\pi,\phi)\notin \mathcal{B}_\pi^\phi$. 
    \\
    \\
    Next, let $i=k$. By definition, $B_k(\pi,\phi)=(e_k,\ldots,e_{p+k-2},e_1,l_{p+k},\ldots, l_{p+s+k-2},l_{2(p+s)-1})$. We also have $R_{p+k-1}(\iota,\omicron)=(r_{p+k-1},\ldots,r_{p+k+s-2},e_{p+s+k-1},\ldots,e_{2p+s+k-2})$. Clearly, $B_k(\pi,\phi)\cap R_{p+k-1}(\iota,\omicron)=\emptyset$, thus $B_k(\pi,\phi)\notin \mathcal{B}_\pi^\phi$.  Now, let $i=1$. If $B_i(\pi,\phi)\in \mathcal{B}_\pi^\phi$, then using Lemma \ref{lem2}, $\mathcal{B}_\pi^\phi=\{B_1(\pi, \phi) \}$ and consequently, we have $R_n(\pi,\phi)\in \mathcal{R}_\pi^\phi$. This is a contradiction since $R_n(\pi,\phi)\cap B_{p+s}(\iota,\omicron)=\emptyset$. 
    \\
    \\
    Next, we consider the case when $i\in [p+s+1,p+s+k-1]$. For any $i$ in this range, $R_{i-(p+s)}(\iota,\omicron)\cap B_i(\pi,\phi)=\emptyset$, thus implying that $B_i(\pi,\phi)\notin \mathcal{B}_\pi^\phi$. 
    \\
    \\
    Finally, we consider the case when $i\in [p+s+k, 2(p+s)-1]$. Let $j\geq 1$ be minimum such that $B_{p+s+k-1+j}(\pi,\phi)\in \mathcal{B}_\pi^\phi$. Since $B_{2(p+s)}(\pi,\phi)\notin \mathcal{B}_\pi^\phi$, we know from Lemma \ref{lem2} that $|\mathcal{B}_\pi^\phi|\leq (2(p+s)-1)-(p+s+k-1+j)+1=p+s-(k+j)+1$. Note that this implies $k'\geq k+j$. Using Lemma \ref{lem2}, if $j'$ is the minimum index such that $R_{j'}(\pi,\phi)\in \mathcal{R}_\pi^\phi$, then $j'= (p+s+k-1+j)-(k'-1)=p+s+k+j-k'\leq p+s$. Consequently, $R_m(\pi,\phi)\in \mathcal{R}_\pi^\phi$, where $m=j'+(p+k'-1)-1\geq p+s+2$. In particular, we have $R_{p+s+1}(\pi,\phi)\in \mathcal{R}_\pi^\phi$, which is a contradiction since $R_{p+s+1}(\pi, \phi)\cap R_{1}(\iota,\omicron)=\emptyset$.
    \\
    \\
    We can now conclude that all of the $B$-sets in $\mathcal{B}_\pi^\phi$ have starting indices in the interval $[k+1,p+s]$. However, $R_{p+k}(\pi,\phi)\notin \mathcal{R}_\pi^\phi$ since $R_{p+k}(\pi,\phi)\cap B_k(\iota,\omicron)=\emptyset$. Recall also that $R_n(\pi,\phi)\notin \mathcal{R}_\pi^\phi$. It follows from Lemma \ref{lem2} that all of the $R$-sets in $\mathcal{R}_\pi^\phi$ have starting indices that lie in $[1, p+k-1]$. Therefore, $|\mathcal{F}_\pi^\phi|\leq (p+s-k)+(p+k-1)\leq 2p+s-1$, a contradiction. 
\end{proof}
\noindent
The two figures below illustrate the two extremal possibilities from Lemma \ref{lem3} for any given cyclic order. In each of the figures, we have $n=12$, $p=3$ and $s=2$, and the cyclic order considered for the illustration is $C_\iota^\omicron$. The portion of each interval that includes singletons (left vertices from the last $s$ positions for blue intervals, right vertices from the first $s$ positions for red intervals) is denoted by dashed lines. For clarity, only the first and last blue and red intervals are drawn. In Figure \ref{kone}, we have $k=1$, and $C_\iota^\omicron$, with $\mathcal{B}=\{B_8,B_9,B_{10},B_{11},B_{12}\}$ and $\mathcal{R}=\{R_{8},R_{9},R_{10}\}$, is centered at $l_n$. Similarly, in Figure \ref{kspone}, $C_\iota^\omicron$ is centered at $r_n$. 
\begin{figure}[H]
\centering
\begin{tikzpicture}[scale=0.5]
    \foreach \x [count=\p] in {0,...,11} {
        \node[shape=circle,draw=black,scale=0.5] (13-\p) at (-\x*30+60:5) {$e_{\p}$};};   
 \draw [blue,very thick,dashed, domain=75:135] plot ({5.6*cos(\x)}, {5.6*sin(\x)});
\draw [blue,very thick,dashed, domain=75:135] plot ({4.4*cos(\x)}, {4.4*sin(\x)});
 \draw [blue,very thick, domain=135:225] plot ({5.6*cos(\x)}, {5.6*sin(\x)});
\draw [blue,very thick, domain=135:225] plot ({4.4*cos(\x)}, {4.4*sin(\x)});
\draw [blue,very thick,dashed, domain=4.4:5.6] plot ({\x*cos(75)}, {\x*sin(75)});
\draw [blue,very thick, domain=4.4:5.6] plot ({\x*cos(225)}, {\x*sin(225)});
\draw [blue] ({6.6*cos(210)},{6.6*sin(210)}) node[label=center:\large $B_{8}$] {}; 
 \draw [blue,very thick,dashed, domain=-45:15] plot ({5.8*cos(\x)}, {5.8*sin(\x)});
\draw [blue,very thick,dashed, domain=-45:15] plot ({4.2*cos(\x)}, {4.2*sin(\x)});
 \draw [blue,very thick, domain=15:105] plot ({5.8*cos(\x)}, {5.8*sin(\x)});
\draw [blue,very thick, domain=15:105] plot ({4.2*cos(\x)}, {4.2*sin(\x)});
\draw [blue,very thick,dashed, domain=4.2:5.8] plot ({\x*cos(-45)}, {\x*sin(-45)});
\draw [blue,very thick, domain=4.2:5.8] plot ({\x*cos(105)}, {\x*sin(105)});
\draw [blue] ({6.7*cos(90)},{6.7*sin(90)}) node[label=center:\large $B_{12}$] {}; 
 \draw [red,ultra thick, dashed, domain=165:227] plot ({6*cos(\x)}, {6*sin(\x)});
\draw [red,ultra thick, dashed, domain=165:227] plot ({4*cos(\x)}, {4*sin(\x)});
 \draw [red,ultra thick, domain=73:165] plot ({6*cos(\x)}, {6*sin(\x)});
\draw [red,ultra thick, domain=73:165] plot ({4*cos(\x)}, {4*sin(\x)});
\draw [red,ultra thick, domain=4:6] plot ({\x*cos(73)}, {\x*sin(73)});
\draw [red,ultra thick, dashed, domain=4:6] plot ({\x*cos(227)}, {\x*sin(227)});
\draw [red] ({3.5*cos(210)},{3.5*sin(210)}) node[label=center:\large $R_{8}$] {}; 

 \draw [red,ultra thick, dashed, domain=105:165] plot ({6.2*cos(\x)}, {6.2*sin(\x)});
\draw [red,ultra thick, dashed, domain=105:165] plot ({3.8*cos(\x)}, {3.8*sin(\x)});
\draw [red,ultra thick, domain=13:105] plot ({6.2*cos(\x)}, {6.2*sin(\x)});
\draw [red,ultra thick, domain=13:105] plot ({3.8*cos(\x)}, {3.8*sin(\x)});
\draw [red,ultra thick, domain=3.8:6.2] plot ({\x*cos(13)}, {\x*sin(13)});
\draw [red,ultra thick, dashed, domain=3.8:6.2] plot ({\x*cos(165)}, {\x*sin(165)});
\draw [red] ({3.2*cos(150)},{3.2*sin(150)}) node[label=center:\large $R_{10}$] {}; 

\end{tikzpicture}
\caption{Cyclic order $C_\iota^\omicron$ centered at $l_n$ (with $k=1$).}
\label{kone}
\end{figure}
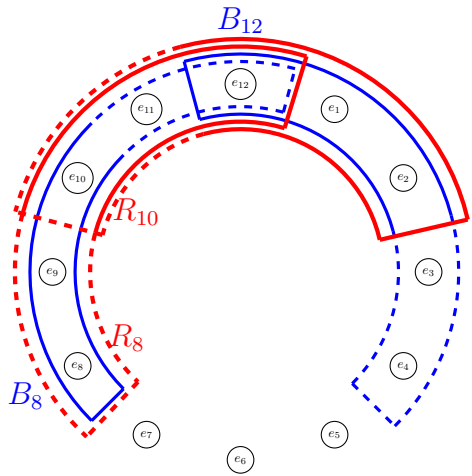

\begin{figure}[H]
\centering
    \begin{tikzpicture}[scale=0.5]
    \foreach \x [count=\p] in {0,...,11} {
        \node[shape=circle,draw=black,scale=0.5] (13-\p) at (-\x*30+60:5) {$e_{\p}$};};   
  \draw [red,very thick, domain=75:167.5] plot ({5.6*cos(\x)}, {5.6*sin(\x)});
\draw [red,very thick, domain=75:167.5] plot ({4.4*cos(\x)}, {4.4*sin(\x)});
 \draw [red,very thick, dashed, domain=167.5:225] plot ({5.6*cos(\x)}, {5.6*sin(\x)});
\draw [red,very thick, dashed, domain=167.5:225] plot ({4.4*cos(\x)}, {4.4*sin(\x)});
\draw [red,very thick, domain=4.4:5.6] plot ({\x*cos(75)}, {\x*sin(75)});
\draw [red,very thick, dashed, domain=4.4:5.6] plot ({\x*cos(225)}, {\x*sin(225)});
\draw [red] ({6.6*cos(210)},{6.6*sin(210)}) node[label=center:\large $R_{8}$] {}; 
 \draw [red,very thick, domain=-45:45] plot ({5.8*cos(\x)}, {5.8*sin(\x)});
\draw [red,very thick, domain=-45:45] plot ({4.2*cos(\x)}, {4.2*sin(\x)});
 \draw [red,very thick, dashed, domain=45:105] plot ({5.8*cos(\x)}, {5.8*sin(\x)});
\draw [red,very thick, dashed, domain=45:105] plot ({4.2*cos(\x)}, {4.2*sin(\x)});
\draw [red,very thick, domain=4.2:5.8] plot ({\x*cos(-45)}, {\x*sin(-45)});
\draw [red,very thick, dashed, domain=4.2:5.8] plot ({\x*cos(105)}, {\x*sin(105)});
\draw [red] ({6.7*cos(90)},{6.7*sin(90)}) node[label=center:\large $R_{12}$] {}; 
 \draw [blue,ultra thick, domain=73:165] plot ({6*cos(\x)}, {6*sin(\x)});
\draw [blue,ultra thick, domain=73:165] plot ({4*cos(\x)}, {4*sin(\x)});
\draw [blue,ultra thick, dashed, domain=13:73] plot ({6*cos(\x)}, {6*sin(\x)});
\draw [blue,ultra thick, dashed, domain=13:73] plot ({4*cos(\x)}, {4*sin(\x)});
\draw [blue,ultra thick, dashed, domain=4:6] plot ({\x*cos(13)}, {\x*sin(13)});
\draw [blue,ultra thick, domain=4:6] plot ({\x*cos(165)}, {\x*sin(165)});
\draw [blue] ({3.4*cos(150)},{3.4*sin(150)}) node[label=center:\large $B_{10}$] {}; 
 \draw [blue,ultra thick, dashed, domain=-48:13] plot ({6.2*cos(\x)}, {6.2*sin(\x)});
\draw [blue,ultra thick, dashed, domain=-48:13] plot ({3.8*cos(\x)}, {3.8*sin(\x)});
 \draw [blue,ultra thick, domain=13:108] plot ({6.2*cos(\x)}, {6.2*sin(\x)});
\draw [blue,ultra thick, domain=13:108] plot ({3.8*cos(\x)}, {3.8*sin(\x)});
\draw [blue,ultra thick, dashed, domain=3.8:6.2] plot ({\x*cos(-48)}, {\x*sin(-48)});
\draw [blue,ultra thick, domain=3.8:6.2] plot ({\x*cos(108)}, {\x*sin(108)});
\draw [blue] ({3.3*cos(90)},{3.3*sin(90)}) node[label=center:\large $B_{12}$] {}; 
\end{tikzpicture}
\caption{Cyclic order $C_\iota^\omicron$ centered at $r_n$ (with $k=s+1$).}
\label{kspone}
\end{figure}
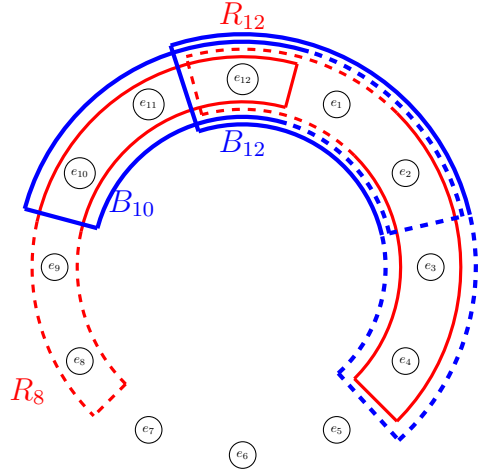
\noindent
We now show that each cyclic order in $\mathscr{C}_n$ is centered at the same vertex. This would prove the uniqueness of the extremal structures for Theorem \ref{th1}. We first assume, without loss of generality (relabeling if necessary), that the canonical cyclic order $C_\iota^\omicron$ is centered at $e_n^0=l_n$. In other words, we have $k=k(\iota,\omicron,p,s)=1$, $\mathcal{B}_\sigma^\tau=\{B_{n-p-s+1},B_{n-p-s+2},\ldots, B_n\}$,and $\mathcal{R}_\sigma^\tau=\{R_{n-p-s+1}, R_{n-p-s+2}\ldots,R_{n-s}\}$. Our final lemma proves that every cyclic order in $\mathscr{C}_n$ is centered at $l_n$. 
\\
\\
\noindent We first make some preliminary observations that will help simplify its proof. For any $C_\sigma^\tau$, we denote its \textit{reflection} by $C_{\overline{\sigma}}^{\overline{\tau}}$, where $\overline{\sigma}$ is the permutation given by $\overline{\sigma}(i)=\sigma(n-i)$ for each $1\leq i\leq n-1$, and $\overline{\tau}$ is the sequence given by $\overline{\tau}_i=\tau_i+1 (\textrm{mod } 2)$ for each $1\leq i\leq n$. It is easy to see that $\mathcal{F}_\sigma^\tau=\mathcal{F}_{\overline{\sigma}}^{\overline{\tau}}$; indeed, we have $\mathcal{B}_\sigma^\tau=\mathcal{R}_{\overline{\sigma}}^{\overline{\tau}}$, and $\mathcal{R}_\sigma^\tau=\mathcal{B}_{\overline{\sigma}}^{\overline{\tau}}$. Additionally, we also have $k(\sigma,\tau,p,s)=(s+2)-k(\overline{\sigma},\overline{\tau},p,s)$. In view of this, we can restrict our attention to only those cyclic orders $C_\sigma^\tau$ with $\tau_n=0$. As mentioned earlier, we will also identify each $\sigma$ as a permutation in $S_{n-1}$. We denote this subset of $\mathscr{C}_n$ by $\mathscr{C}'(n)$. Clearly $C_\iota^\omicron\in \mathscr{C}'(n)$.

\begin{lem}\label{lem4}
    Let $C_\sigma^\tau\in \mathscr{C}'_n$ be centered at vertex $l_n$. Then:    
    \begin{enumerate}
        \item \label{tp} For each $1\leq i\leq n-2$,  $C_{\pi_i}^\phi$ is centered at $l_n$.
        \item \label{sp} There exists $1\leq j\leq n-1$ such that $C_\sigma^{\tau^j}$ is centered at $l_n$.
    \end{enumerate}
\end{lem}
\noindent
Before we proceed to the proof, we note that the $n-2$ adjacent transpositions and any one of the $n-1$ swaps generate all of the cyclic orders in $\mathscr{C}'(n)$. \footnote{For swaps $s_i$ and $s_j$ with $i<j$, it is clear to see that $s_i=t_{i}\circ t_{i+1}\circ \cdots \circ t_{j-2} \circ t_{j-1}\circ s_j \circ t_{j-1} \circ t_{j-2}\circ \cdots \circ t_{i+1}\circ t_{i}$. Note that the standard multiplication convention is adopted here, i.e., the operations are applied in order from right to left.} Lemma \ref{lem4} thus immediately implies that every cyclic order in $\mathscr{C}'(n)$ is centered at $l_n$.
\begin{proof}[Proof of Lemma \ref{lem4}]

It suffices to prove the lemma for the canonical cyclic order $C_\iota^\omicron$. Note that $B_{n-p-s+1}(\iota,\omicron)\cap B_n(\iota,\omicron)=\{l_n\}$. We first prove Part \ref{tp} of the lemma. 
\\
\\
Let $1\leq i\leq n-2$. The following observation will be frequently used: For any $C^\sigma_\tau$ with $|\mathcal{F}_\sigma^\tau|=2p+s$, if there exist $B, B'\in \mathcal{B}_\sigma^\tau$ with $B\cap B'=\{x\}$, then $C^\sigma_\tau$ is centered at $x$. The proof now splits into cases, depending on the value of $i$.

\begin{enumerate}
    \item Let $i \notin \{ p+s-1, n-p-s\}$. Then $B_{n-p-s+1}(\pi_i,\phi)=B_{n-p-s+1}(\iota,\omicron)$ and $B_n(\pi_i,\phi)=B_n (\iota,\omicron)$. Using the observation from above, we can conclude that $C_{\pi_i}^\phi$ is centered at $l_n$.
    \item Let $i=p+s-1$. Then $B_j(\pi_i,\phi)=B_j(\iota,\omicron)$ for each $n-p-s+1\leq j\leq n-1$. Similarly, $B_{n-p-s}(\pi_i,\phi)=B_{n-p-s}(\iota,\omicron)$ and $R_{n-p-s}(\pi_i,\phi)=R_{n-p-s}(\iota,\omicron)$ since $n>2(p+s)$; thus, neither $B_{n-p-s}(\pi_i,\phi)$ nor $R_{n-p-s}(\pi_i,\phi)$ are in $\mathcal{B}_{\pi_i}^\phi$ and $\mathcal{R}_{\pi_i}^\phi$, respectively, as they are both disjoint from $B_n(\iota,\omicron)$. This implies that $k(\pi_i,\phi,p,s)=1$, i.e., $B_n(\pi_i,\phi)\in \mathcal{B}_\pi^\phi$. Thus, $C_{\pi_i}^\phi$ is centered at $l_n$.
    \item Let $i=n-p-s$. In this case, $B_j(\pi_i,\phi)=B_j(\iota,\omicron)\notin \mathcal{B}_{\pi_i}^\phi$ for $j=1$ and $B_j(\pi_i,\phi)=B_j(\iota,\omicron)\in \mathcal{B}_{\pi_i}^\phi$ for $n-p-s+2\leq j\leq n$. Suppose that $B_{n-p-s+1}(\pi_i,\phi)\notin \mathcal{B}_{\pi_i}^\phi$. Then $k(\pi_i,\phi,p,s)=2$, which by Lemma \ref{lem3} implies that $s=1$. However, Lemma \ref{lem2} implies that $R_n(\pi_i,\phi)\in \mathcal{R}_{\pi_i}^\phi$, a contradiction, as $R_n(\pi_i,\phi) \cap B_{n-p-s+1}(\iota,\omicron)=\emptyset$.
\end{enumerate}
\noindent
We now prove Part \ref{sp} of the lemma. For this, we choose $j=p+s$; indeed, any $j\in [p+s,n-(p+s)]$ suffices. Clearly, $C_\sigma^{\tau^j}$ is still centered at $l_n$. 
\end{proof}
\section{Proof of Theorem \ref{th2}}\label{SOne}
We begin by proving Part \ref{two1} of Theorem \ref{th2}. Let $p, s\geq 1$, and let $\mathcal{F}\subseteq \mathcal{H}^{(p,s)}(2p+s)$ be intersecting. Now, consider any subgraph $H\in \mathcal{H}^{(p,s)}(2p+s)$. Clearly, the subgraph $H'$ induced by the vertex set $V(M_n)\setminus V(H)$, is also a member of $\mathcal{H}^{(p,s)}(2p+s)$. Thus, at most one of $H$ and $H'$ can be in $\mathcal{F}$. This immediately yields the upper bound $|\mathcal{F}|\leq \dfrac{1}{2}|\mathcal{H}^{(p,s)}(2p+s)|$, and using Equation \ref{eq3}, this simplifies to $|\mathcal{F}|\leq |\mathcal{H}_x^{(p,s)}(n)(2p+s)|$ (where $x$ is any vertex from $V(M_n)$), completing a proof of Part \ref{two1} of the theorem. 
\\
\\
To see that $\mathcal{H}^{(p,s)}(2p+s)$ is not strongly EKR, we construct an intersecting subfamily of $\mathcal{H}^{(p,s)}(2p+s)$ that has extremal size, but is not a star. Let $\mathcal{G}=\mathcal{H}^{(p,s)}(2p+s)\setminus \mathcal{H}^{(p,s)}_{l_{2p+s}}(2p+s)$ be the family of all subgraphs that do not contain the vertex $l_{2p+s}$. It is clear to see that $\mathcal{G}$ has extremal size but is not a star, so we only need to prove that it is intersecting. By way of contradiction, suppose $G_1,G_2\in \mathcal{G}$ and $G_1\cap G_2=\emptyset$. Without loss of generality, suppose that $G_1$ contains (both endpoints from) each of the $p$ edges $e_1,e_2,\ldots,e_{p}$, while $G_2$ contains each of the $p$ edges $e_{p+1},\ldots,e_{2p}$. Thus, the $s$ singletons in both $G_1$ and $G_2$ must be from the $s$ edges $e_{2p+1},\ldots,e_{2p+s}$. By definition of $\mathcal{G}$, this implies $r_{2p+s}\in G_1\cap G_2$, a contradiction. 
\\
\\
We proceed to Part \ref{two2} of Theorem \ref{th2}. Let $s=1$. We know from Theorem \ref{th1} that $\mathcal{H}^{(p,1)}(n)$ is EKR when $n\geq 2p+2$ and strongly EKR when $n>2p+2$. Additionally, we know from Part \ref{two1} of Theorem \ref{th2} that $\mathcal{H}^{(p,1)}(n)$ is EKR for $n=2p+1$. We now prove that it is strongly EKR when $n=2p+2$. Let $\mathcal{F}\subseteq \mathcal{H}^{(p,1)}(2p+2)$ be intersecting, and let $|\mathcal{F}|=(2p+1)\binom{n-1}{p}$.
\\
\\
We use a more standard Katona-type argument to prove this result. In particular, we only consider those cyclic orders $C_\sigma=C_\sigma^\omicron\in \mathscr{C}_n$ where $\sigma\in S_n$ with $\sigma(n)=n$ and $\omicron\in \{0,1\}^n$ is the all-zeros sequence of length $n$. Clearly there are $(n-1)!$ such cyclic orders. Additionally, we can interpret a given cyclic order $C_\sigma$ as a permutation of the vertex set $V(M_n)$. As an example, for $n=6$ and $\sigma=(\sigma(1),\ldots,\sigma(6))=(5,3,2,1,4,6)$, the corresponding cyclic order is $(l_5,r_5,l_3,r_3,l_2,r_2,l_1,r_1,l_4,r_4,l_6,r_6)$. Note that any sequence containing $2p+1$ consecutive vertices from a cyclic order $C_\sigma$ corresponds to a subgraph in $\mathcal{H}^{(p,1)}(n)$. This correspondance allows us to directly apply many of the ideas from Katona's proof of the EKR theorem \cite{Katona}. 

For a given cyclic order $C_\sigma$, a subgraph $H \in \mathcal{H}^{(p,1)}(n)$ is an \textit{interval} in $C_\sigma$ if either
$$H=\{r_{\sigma(i)}, e_{\sigma(i+1)}, \ldots, e_{\sigma(i+p)} \},$$ or 

$$H=\{e_{\sigma(i)},\ldots, e_{\sigma(i+p-1)},  l_{\sigma(i+p)} \},$$
for some $i \in [n]$, where addition is carried out modulo $n$. For a given $\sigma$, we refer to the former type as an \textit{$r$-interval} at position $i$, and the latter as the \textit{$l$-interval} at position $i$. Let $\mathcal{F}_\sigma$ be the set of all subgraphs in $\mathcal{F}$ that are intervals in $C_\sigma$. 
\\
\\
Since $|\mathcal{F}|$ has maximum size, it implies that for each cyclic order $C_\sigma$, $|\mathcal{F}_\sigma|=2p+1$. Using the corollary of Katona's lemma (cited in the proof of Lemma \ref{lem2} in Section \ref{characterization}), there exists $x\in V(M_n)$ such that $C_\sigma$ is centered at $x$, i.e., $x\in \bigcap\limits_{F\in \mathcal{F}_\sigma} F$.
\\
\\
To complete the proof, we focus on the only remaining case $n=2p+2$; without loss of generality, assume that the permutation $\iota=(1, 2, \ldots, 2p+2)$ is centered at $r_{2p+2}$. As in the proof of Lemma \ref{lem4} from Theorem \ref{th1}, to show that $\mathcal{F}=\mathcal{H}^{(p,1)}_{r_{2p+2}}(2p+2)$, it suffices to show that for $\sigma=\iota\circ t_i$, where $1\leq i \leq 2p$, $C_\sigma$ is centered at $r_{2p+2}$. Let $H=\{r_{2p+2}, e_1, \ldots, e_p\}$, and $K=\{r_{p+2},e_{p+3},\ldots,e_{2p+2}\}$. Note that $H, K\in \mathcal{F}_\iota$, and $H\cap K=\{r_{2p+2}\}$.
The proof now proceeds through cases, depending on the value of $i$. In each case, we prove that in $C_\sigma$, there are two intervals in $\mathcal{F}_\sigma$ that intersect only in $r_{2p+2}$. The only non-trivial cases to consider are when $i\in \{p,p+1,p+2\}$, as for any other value of $i$, both $H$ and $K$ are unchanged by the adjacent transposition $t_i$. 
\begin{enumerate}
    \item Let $i=p$. Clearly, the interval $K$ is unchanged by the transposition $t_p$. However, the $l$-interval at position $p+2$ is also unchanged by the transposition and is thus not a member of $\mathcal{F}_\sigma$. Since $|\mathcal{F}_\sigma|=2p+1$, this implies that the $s$-interval at position $2p+2$, say $H'$ is in $\mathcal{F}_\sigma$. Since $K\cap H'=\{r_{2p+2}\}$, $C_\sigma$ is centered at $r_{2p+2}$. 
    \item Let $i=p+1$. In this case, the interval $H$ is unchanged by the transposition. We can also argue that the $l$-interval at position $1$, say $H'$, is not in $\mathcal{F}_\sigma$. By way of contradiction, suppose that $H'\in \mathcal{F}_\sigma$. Since $H'=\{e_1,\ldots,e_p,l_{p+2}\}$, clearly $H'\cap K=\emptyset$, a contradiction. This implies that the $r$-interval at position $p+2$, say $K'$, is in $\mathcal{F}_\sigma$. Since $H\cap K'=\{r_{2p+2}\}$, $C_\sigma$ is centered at $r_{2p+2}$. 
    \item Let $i=p+2$. As in the previous case, the interval $H$ is unchanged by $t_i$. However, the $l$-interval at position $1$ is also unchanged by $t_i$ and is thus not a member of $\mathcal{F}_\sigma$. This implies that the $r$-interval at position $p+2$, say $\tilde{K}$, is in $\mathcal{F}_\sigma$. Since $H\cap \tilde{K}=\{r_{2p+2}\}$, $C_\sigma$ is centered at $r_{2p+2}$. \qedsymbol
\end{enumerate}

\section{Future directions} \label{future}
It is evident from the proof of Theorem \ref{th1} that a stronger condition on $n$ (in comparison to the one proposed by Conjecture \ref{cj}) is required for the use of Katona's cycle method. It is currently unclear to us if the cycle method can be used to settle the conjecture completely. The algebraic framework for applying Katona's cycle method described in \cite{BorgMea} would be an interesting possibility to consider; shifting/compression techniques are another. 
\\
\\
We also propose the following general formulation of the problem that can potentially lead to other, natural extensions of the EKR theorem and its variants. For positive integers $m$ and $n$, let $G_{m,n}$ be a graph that has exactly $m$ components, each of which is a copy of $K_n$, the complete graph on $n$ vertices. Consider a sequence of non-negative integers $\mathbf{s}=(s_1,\ldots,s_n)$ with $\sum_{i=1}^n s_i\leq m$. If $H$ is an induced subgraph of $G_{m,n}$, we say that $H$ has signature $\mathbf{s}$ if it satisfies the following two conditions:
\begin{enumerate}
    \item $H$ contains exactly $\sum_{i=1}^n s_i$ components, with exactly $s_i$ copies of $K_i$ for each $1\leq i\leq n$. 
    \item Let $H_1,H_2$ be any two distinct components of $H$. If $H_1\subseteq G_1$ and $H_2\subseteq G_2$, where $G_1$ and $G_2$ are components of $G_{m,n}$, then $G_1\neq G_2$.
\end{enumerate}
Let $\mathcal{H}^{\mathbf{s}}(m,n)$ be the family of all induced subgraphs of $G_{m,n}$ with signature $\mathbf{s}$.
\begin{prob}
    For a signature sequence $\mathbf{s}$ and positive integers $m$ and $n$, find a best possible function $f(m,\mathbf{s})$ such that for $n\geq f(m,\mathbf{s})$, $\mathcal{H}^{\mathbf{s}}(m,n)$ is EKR.
\end{prob}


\end{document}